\documentclass{amsart}

\usepackage{amsmath}
\usepackage{amsfonts}
\usepackage{amssymb}
\usepackage{bm}
\usepackage{amscd}
\usepackage{amsthm}
\usepackage{setspace}
\usepackage{graphicx}
\usepackage{hyperref}
\usepackage[shortlabels]{enumitem}

\setenumerate{leftmargin=*}

\newcommand{\forceindent}{\leavevmode{\parindent=1em\indent}}

\newtheorem{theorem}{Theorem}[section]
\newtheorem{lemma}[theorem]{Lemma}
\newtheorem{corollary}[theorem]{Corollary}
\newtheorem{proposition}[theorem]{Proposition}
\newtheorem{claim}[theorem]{Claim}

\newtheorem{assumption}[theorem]{Assumption}

\theoremstyle{definition}
\newtheorem{remark}[theorem]{Remark}
\newtheorem{definition}[theorem]{Definition}

\newtheorem{example}[theorem]{Example}
\newtheorem{note}[theorem]{Note}

\newcommand{\Q}{\mathbb{Q}}

\newcommand{\N}{\mathbb{N}}

\newcommand{\Z}{\mathbb{Z}}

\newcommand{\m}{\mathfrak{m}}
\newcommand{\n}{\mathfrak{n}}
\newcommand{\p}{\mathfrak{p}}
\newcommand{\gr}{\text{gr}}
\newcommand{\ann}{\text{Ann}}
\newcommand{\spec}{\text{Spec}}

\newcommand{\rr}{\mathbf{R}}

\newcommand{\mysetminus}{\bm{\smallsetminus}}

\DeclareMathOperator{\h}{ht}

\begin{document}

\title[Nonassociative complete  filtered rings]{The Artin--Rees lemma and size of spaces over  nonassociative complete filtered rings}
\author{Jayanta Manoharmayum }
\address{School of Mathematics and Statistics, University of
    Sheffield,
    Sheffield S3 7RH,
    U.K.}
\email{J.Manoharmayum@sheffield.ac.uk}

\date{}
\subjclass[2010]{17A01, 17A60, 17D99}

\begin{abstract} This paper studies nonassociative filtered rings using associated gradations. We show that a complete filtered ring $R$ with affine associated graded ring  generated in degree $1$ is a local ring, and prove that the Artin--Rees lemma holds for $R$. Assuming finiteness of the residue field, we derive asymptotics for  abelian groups with an operation of  $R$, and identify classes of torsion elements for spaces over a central extension of $R$.
\end{abstract}

\maketitle

\section{Introduction}\label{S:introduction}

A particularly effective technique in the study of filtered rings is to exploit the natural synergy between the filtration and its associated graded ring.  In many instances, the associated graded ring is a  simpler object to understand and more tractable; in the presence of a good filtration, we can utilise this knowledge to derive results about the original ring.  A striking demonstration of this strategy   is provided by the work of Huishi and van Oystaeyen on Zariskian filtrations in the setting of non-commutative rings. (For details, see \cite{HvOCommAlg}, \cite{HvO}; for a different application, see \cite{StaffordZhang}.). 
This paper illustrates that the  idea of linking  filtered and graded problems can be effective even in the feral world of nonassociative rings.

Before proceeding to a description of the main results, we record our conventions for rings and filtrations. First, all rings are assumed to be unital. An unqualified ring always refers to  a  nonassociative (that is, not necessarily associative) ring; by contrast, unless indicated otherwise, a commutative ring/algebra is automatically assumed to be associative. Second,
all filtrations and gradings that we consider are assumed to be indexed over the  monoid $\N_0=\{0,1,2,\dots\}$. We will use straight-forward adaptations of terminology from the associative setting in this section; formal descriptions are given in later sections.

 The results of this paper are obtained in the setting of complete filtered rings with affine associated graded rings which are generated in degree $1$. This is a broad class  of rings, encompassing Iwasawa algebras and complete local Noetherian commutative rings.  In detail, and for the remainder of this section, we assume the following.
\begin{enumerate}[wide=0pt]
 \item[]\emph{$R$ is a  ring equipped with a complete filtration $(F^*(R))$ satisfying the following properties: the quotient $F^0(R)/F^1(R)$ is a field, and the associated graded ring $\gr(R)$ is a commutative algebra over  $F^0(R)/F^1(R)$, finitely generated in degree $1$. }\end{enumerate} 
 We will then show that the following assertions hold.
\begin{enumerate}[label=(A\arabic*) , font=\bfseries,  wide=0pt, itemsep=1pt]
\item \label{A1}The ring $R$ is a local ring with maximal ideal $F^1(R)$. We set $\m:=F^1(R)$.
\item \label{A2}The equality $F^i(R)F^j(R)=F^{i+j}(R)$ holds for all $i,j\in \N_0$. In particular, the ideals $\m^2,\m^3,\m^4,\dots $ are defined unambiguously, and the filtration on $R$ is the $\m$-adic filtration.
\item \label{A3} In both the ring $R$ and the Rees ring  associated to the filtration $(F^*(R))$,  every left ideal is the linear span of  finitely many elements. Thus $(F^*(R))$ can be viewed as a  Zariskian filtration for nonassociative rings (cf. \cite[Chapter II \S 2]{HvO}). The rationale for avoiding the term \emph{Noetherian} here will become transparent in Section \ref{ss:gen}, after Definition \ref{D:fin span}.
\item \label{A4} If $I$ is a left ideal of $R$, then there exists an integer $D\in \N_0$ such that 
\[\m^{n+d}\cap I=\m^n(\m^d\cap I)\]
for all integers $n\in\N_0$ and integers $d\geq D$. The statement is a nonassociative version of the Artin--Rees lemma (cf. \cite[Theorem 8.5]{mats}).
 \end{enumerate}

Assertions \ref{A1}, \ref{A2} and \ref{A3} are covered by  Theorem \ref{T:CLF}.  Assertion \ref{A4} is covered by Theorem \ref{T:AR3} in the wider setting of \emph{$R$-spaces} (which we discuss  in Section \ref{ss:gen}). For the purposes of this introduction,  we need only keep in mind the following remarks:   an $R$-space is an abelian group on which the ring $R$ operates bilinearly and unitally; the class of $R$-spaces includes ideals, and $R$-modules when $R$ is associative. 

The hypotheses on the associated graded ring allow us to define dimensions of filtered $R$-spaces independently of the filtration (see Proposition \ref{P:kdim}). When the residue field $R/\m$ is finite, we can use the associated dimension theory to obtain asymptotics for sizes of $R$-spaces. The general case is covered by Theorem \ref{T:asymptotic}; we give a  formulation for ideals here.
\begin{enumerate}[label=(A\arabic*) , font=\bfseries,  wide=0pt, itemsep=1pt]\setcounter{enumi}{4}
\item \label{A5} Assume that $R/\m$ is finite. Let   $q$ be the cardinality of $R/\m$, and let $I$ be a left ideal of $R$. Then 
\[\log_q\vert I/\m^nI\vert\sim \alpha n^\delta\]
for some  non-negative integer $\delta\in \N_0$ and a rational number $\alpha\in \Q$.\end{enumerate}
(Note that the abelian groups $I/\m^nI$  in assertion \ref{A5} above need not be $R$-spaces.)

As an application, we  consider torsion elements for  spaces over central extensions. 
The set up is as follows. Assume that the residue field  $R/\m$ is finite and let $R[[T]]$ be the ring of formal power series in a central indeterminate $T$.  Given an $R[[T]]$-space $M$ satisfying some conditions as an $R$-space, we identify classes of elements in $M$ which are killed by $R$. For  associative rings, the result relates torsion modules to pseudo-nullity over associated graded rings. The application is motivated by the relationship between torsion and pseudo-nullity for modules over Iwasawa algebras (see \cite{css}, \cite{hachimori}, \cite{howson}, \cite{ochi-venjakob}).

The rest of the  paper is organised as follows. In Section \ref{S:FC}, we set up the  terminology of  $R$-spaces  and finiteness conditions in the setting of $R$-spaces. Section \ref{S:FR} reviews some of the standard facts on filtered abelian groups and examines how finiteness conditions on associated gradations can be lifted to the filtered ring and filtered spaces.  Finiteness of the Rees ring and the Artin--Rees lemma in the setting of $R$-spaces are examined in Section \ref{S:AR} and Section \ref{S:DPS}, leading to the main results covering the assertions  \ref{A1} to \ref{A5} above. Finally, Section \ref{S:TPS} covers application to  torsion and pseudo-nullity of spaces  under central extensions.

\section{Finiteness conditions for spaces over a ring}\label{S:FC}

\subsection{Generalities}\label{ss:gen} 
 Let $R$ be a unital nonassociative ring.   By a (left) \emph{space over $R$}, or simply an \emph{$R$-space}, we  mean an abelian group $M$ together with a bilinear pairing of additive groups $R\times M\to M$ which is also unital in the sense that $(1,m)\to m$ for all $m\in M$. The $R$-space pairing will always be written  multipicatively:  thus $(r,m)\to rm$ for all elements $r\in R$ and $m\in M$; bilinearity then means that the equalities $(r_1+r_2)m=r_1m+r_2m$ and $r(m_1+m_2)=rm_1+rm_2$ hold for all $r,r_1,r_2\in R$ and $m,m_1,m_2\in M$. 
 
 Subspaces, quotients and morphisms of $R$-spaces are defined in the obvious way, and (can be checked to) have the expected properties. 
Two equivalent descriptions of spaces over rings are noteworthy here. First, an $R$-space $M$ is \emph{an abelian group $M$ with operators in $R$} (see \cite[Ch. I, \S 4.2]{bourbaki}) such that the induced map from $R\to \text{End}(M)$ is a morphism of abelian groups and the element $1\in R$ operates as the identity endomorphism on $M$. Second, and this follows from the universality of the tensor algebra, the category of $R$-spaces is equivalent to the category of
(left) $T(R)/I$-modules where 
 \[T(R):=\Z\,\oplus \, R\,\oplus \, (R\otimes R)\,\oplus \, (R\otimes R\otimes R)\,\oplus \, \dots\]
 is the tensor algebra of the $\Z$-module $R$, and $I$ is the ideal of $T(R)$ generated by $1_R-1_{\Z}$.

 \begin{remark}\label{R:spaces}  It would be desirable to have `modules' rather than spaces, so that  ring multiplication is involved. This can be done using the tensor algebra for rings satisfying polynomial identities, or from a variety (in the sense of universal algebras).  See \cite[Ch. 11, \S 1]{zsss}) for details.  Note that 
even left  ideals---which are left spaces and have ring multiplication built in---need to be appended before they can be used effectively as in the associative setting. (See \cite{lawver}).   \end{remark}

Now let $M$ be an   $R$-space.  As in usual module theory, we say that  $M$ is \emph{(left) Noetherian} if the ascending chain condition holds for $R$-subspaces of $M$:  that is, if $M_1\subseteq M_2\subseteq \dots$ is a sequence of $R$-subspaces of $M$, then we have $M_k=M_{k+1}=\dots $ for some integer $k$. The ring $R$ is left Noetherian if the ascending chain condition holds for left ideals.  

We now consider finiteness conditions on spaces over a ring. Let $M$ be an $R$-space. Recall that if $\Delta$ is a subset of  $M$, then 
the \emph{subspace generated by $\Delta$} is the smallest subspace of $M$  containing $\Delta$. A subspace of $M$ is  \emph{finitely generated} if it is generated by a finite subset of $M$. Unfortunately, finite generation of spaces is a weaker condition than finite generation of associative modules;  this difference is highlighted in the following definition.

\begin{definition}\label{D:fin span}
Let $M$ be an $R$-space.
\begin{enumerate}[(1)]
\item \label{dfs1} Let  $A$ and $\Delta$ be   subsets of $R$ and $M$, respectively. The \emph{product }  $A\Delta$ is  the abelian subgroup of $M$ consisting of all finite sums  
$\sum ax$ with $a\in A$ and $x\in \Delta$. The product $R\Delta$ is the 
 the \emph{$R$-span} of $\Delta$.
 \item \label{dfs2}We say that $M$ is  \emph{finitely spanned} if it is the $R$-span of a finite set: that is, we can find finitely many elements $x_1,\dots , x_n$ in $M$ such that $M=Rx_1+\dots +R x_n$.
\item \label{dfs3} We say that $M$ is  \emph{strongly Noetherian} if every  $R$-subspace of $M$ is finitely spanned.
The ring $R$ is strongly (left) Noetherian if every (left) ideal of $R$ is finitely spanned.
\end{enumerate}
\end{definition}

Note that if $M$ is an $R$-space, then the $R$-subspace of $M$ generated by a subset $\Delta\subseteq M$  is $R\Delta+R(R\Delta)+R(R(R\Delta))+\dots$. Also, as the ring $R$ is unital, we have
\[ R\Delta+R(R\Delta)+R(R(R\Delta))+\dots=\bigcup R(R(\dots (R\Delta)\dots)).\] This  can well fail to be finitely spanned (e.g., consider free nonassociative rings).

\begin{remark}\label{R:noeth}  It is clear that  a strongly Noetherian ring/space is Noetherian. An  associative ring is strongly Noetherian if and only if it is Noetherian.   More generally, suppose the $R$-space $M$ has the following property: for all elements $m\in M$, the $R$-span $Rm$ is a subspace of $M$. Then $M$ is strongly Noetherian if and only it is Noetherian.\end{remark}

\begin{proposition}\label{P:noeth} Let $R$ be a ring and let $M$ be an $R$-space. Then the following statements hold.
\begin{enumerate}[(1)]  
\item \label{noeth1}The space $M$ is Noetherian if, and only if, every $R$-subspace of $M$ is finitely generated.
\item \label{noeth2} Let $L\subseteq M$ be an $R$-subspace. Then $M$ is Noetherian (resp. strongly Noetherian) if and only if both $L$ and $M/L$ are Noetherian (resp. strongly Noetherian).
\end{enumerate}
\end{proposition}

\begin{proof} Proposition \ref{P:noeth} is proved in exactly the same way as in the setting of associative rings and modules, and so we omit the proof. To illustrate one case, let us show that if 
 $L$ and $M/L$ are both strongly  Noetherian then $M$ must be strongly Noetherian. 
 
 So let $N\subseteq M$ be an $R$-subspace.  Now,  both $L\cap N$ (being a subspace of $L$) and $N/(L\cap N)$ (viewed as a subspace of $M/L$) are finitely spanned. If we now take
 \[x_1,\dots x_n,y_1,\dots  y_m\in N\] such that $x_1,\dots x_n$ span $(L\cap N)$ and the images of $y_1,\dots , y_m$ span $N/(L\cap N)$, then $N$ is spanned by $x_1,\dots x_n,y_1,\dots , y_n$.\end{proof}

\begin{remark}
Let $R$ be  a strongly Noetherian  ring.   Then, by Proposition \ref{P:noeth}\ref{noeth2}, we obtain a good supply of strongly Noetherian $R$-spaces by taking  subquotients of finite free $R$-spaces and their extension classes. Note that the construction produces spaces with additional structure. In particular, if $R$ is associative, then we obtain precisely the class of  finitely generated $R$-modules from the construction. \end{remark}

\subsection{Hilbert's basis theorem}\label{ss:prhm} 
We now consider persistence of finiteness conditions under  central extensions. 
If  $R$ is a ring, then  $R[X]$ denotes the polynomial ring in a  central indeterminate $X$; if $M$ is an $R$-space, then  we write $M[X]:=M\oplus MX \oplus M X^2\oplus \dots$ for the $R[X]$-space of polynomials in $X$ with coefficients from $M$.

\begin{proposition}[Hilbert's basis theorem]\label{P:hbt} Let $R$ be a ring and let $M$ be an $R$-space. Then the following statements hold.
\begin{enumerate}[(1)]
\item \label{hbt1} If $M$ is Noetherian, then $M[X]$ is a Noetherian  $R[X]$-module. 
\item  \label{hbt2} If $M$ is strongly Noetherian, then $M[X]$ is a strongly Noetherian  $R[X]$-module. 
\end{enumerate}\end{proposition}

Both parts   of Proposition \ref{P:hbt} can be verified following standard proofs in  the associative setting. However,   a significant number of checks are needed to ensure that  arguments do carry over; therefore, we have  included justifications for completeness.  For Hilbert's basis theorem in a different setting, see \cite{BR}.

We start  with the following construction. Let $N$ be an $R[X]$-subspace of $M[X]$. For each integer $j\in \N_0$, we define a subset   $N(j)\subseteq M$ by
\begin{equation}\label{eq:LT}
 N(j):=\left\lbrace\, m\in M\;\middle|\;
  \begin{tabular}{c}
   $N$ contains a polynomial of the form\\
   $mX^j+\text{lower degree terms}$
   \end{tabular}
 \, \right\rbrace. 
\end{equation}
It is then  easy to check that $N(j)$ is  an $R$-subspace of $M$, and that $N(j)\subseteq N(j+1)$ for all $j\in \N_0$.

\begin{proof}[Proof of Proposition \ref{P:hbt}\ref{hbt1}] Let  $ N_0\subseteq N_1\subseteq N_2\subseteq \dots$ be an ascending chain  
of $R[X]$-subspaces of $M[X]$. For each subspace $N_i$ in the chain,   let $N_i(j)\subseteq M$ be the $R$-subspace of leading coefficients of degree $j$ polynomials in $N_i$  (as in equation \eqref{eq:LT} above). We then have
$N_i(j)\subseteq N_{i'}(j')$ whenever $i\leq i'$ and $j\leq j'$.

Since $M$ is Noetherian, the sequence $N_0(0)\subseteq N_1(1)\subseteq N_{2}(2)\subseteq \dots $
stabilises, say from $n_1$ onwards.  That is, we have $N_j(j)=N_{n_1}(n_1)$ for all $j\geq n_1$.  Also for each non-negative integer $i<n_1$, the chain 
$ N_{i}(0)\subseteq N_{i}(1)\subseteq N_{i}(2)\subseteq \dots $
stabilises. We can therefore find an $n_2\in \N$ such $N_{i}(j)=N_{i}(n_2)$ for all $0\leq i<n_1$ and $n_2\leq j$. 
Take $n=\max\{n_1, n_2\}$. Then we have $N_{i}(j)=N_{n}(j)$ for all $i\geq n$ and $j\in \N_0$. 

So we are reduced to showing that if $L\subseteq N$ are subspaces of $M[X]$ such that $L(j)=N(j)$ for all $j\in \N_0$, then $L=N$. 

Suppose this last claim is not true. Take $\lambda\in N\mysetminus L$ of minimal degree, say $d$, and write $\lambda=aX^d+\text{lower degree terms}$. Thus $a\in N(d)$ and, as $N(d)=L(d)$, we can find an element $\widetilde{\lambda}\in L$ of degree $d$ and leading coefficient $a$. By minimality we must have $(\lambda-\widetilde{\lambda})\in L$, and so $\lambda\in L$---a contradiction. This completes the proof of  part \ref{hbt1}.\end{proof}

\begin{proof}[Proof of Proposition \ref{P:hbt}\ref{hbt2}] 
Let $N$ be a subspace of $M[X]$.   As in equation \eqref{eq:LT},  let $N(j)$ be the subspace of leading coefficients of degree $j$ polynomials in $N$.
Then $N(0)\subseteq N(1)\subseteq N(2)\subseteq \dots$ is an increasing sequence of $R$-subspaces of $M$.  Since a strongly Noetherian space is Noetherian, we must have  $N(k)=N(k+1)=\dots$ for some $k$.

Let $\Delta=\{a_1,\dots ,a_d\}\subseteq N(k)$ span the $R$-space $N(k)$.   For each index $i$ in the set $\{1,\dots ,d\}$, we choose an element $\lambda_i\in N$ as follows: if $a_i\neq 0$, we require  $\lambda_i\in N$ to have  degree $k$ and leading coefficient $a_i$; if $a_i=0$, then  $\lambda_i=0$.

Now consider an element $\lambda \in N$. Let us first suppose that $\lambda$  has degree $n$ with $n\geq k$, and set $a$ to be the leading coefficient of $\lambda$. Now $a\in N(k)$ and so $a=r_1a_1+\dots+r_da_d$ for some $r_1,\dots ,r_d\in R$. Thus 
if we take 
\[\widetilde{\lambda}:= (r_1X^{n-k})\lambda_1+\dots +(r_dX^{n-k})\lambda_d,\] then $\widetilde{\lambda}\in N$ has  degree $n$ and  leading term  $a$. The difference $\lambda-\widetilde{\lambda}$ therefore is an element of $N$ of degree strictly less than the degree of $\lambda$. We can continue this process till we arrive to an element in $N$ of degree less than $k$.

We now suppose that $\lambda\in N$ has degree $l$ with $l<k$. We can then repeat the above argument, this time lifting generators of $N(l)$. And continue, getting down to $N(0)$.

The upshot of all this is as follows. Throwing in extra elements if needed, we find a positive integer $d$ such that the following holds.
For each integer $i\in\{0,\dots , k\}$, we can choose $\lambda_{i1},\dots ,\lambda_{id}\in N$ of the form
\[ \lambda_{ij}=a_{ij}X^i+\text{lower degree terms},\quad j=1,\dots, d,\]  such that $\{\, a_{ij}\mid j=1,\dots , d\}$ spans $N(i)$. Then $N$ is spanned by the 
finite set $\{\, \lambda_{ij}\mid i=0,\dots, k, j=1,\dots ,d\,\}$.\end{proof}

\section{Filtered rings and spaces}\label{S:FR}

In this section, we set out  notation and terminology for filtrations and associated gradations, and consider finiteness conditions (from Section \ref{S:FC}) in the context of filtered spaces. In Section \ref{ss:fag}, we recall some  properties of filtered abelian groups 
 and discuss two lemmas (Lemma \ref{L:filab convergence} and Lemma \ref{L:filabmain}) that underpin many of the constructions in this paper. Specialising  to rings and spaces over rings, we show that the property of being strongly Noetherian can be deduced from properties of   associated gradations (Proposition \ref{P:fsn}) in Section \ref{ss:filring}.

\subsection{Filtered abelian groups}\label{ss:fag} As indicated in Section \ref{S:introduction}, we shall only consider decreasing filtrations  indexed over $\N_0=\{0,1,2,\dots\}$. We will denote the accompanying decreasing filtration on a filtered object by $(F^i(-))$ when no confusion can arise. All abelian groups considered are additive.

By a \emph{filtered abelian group}, we mean an abelian group $A$ together with an exhaustive decreasing filtration $(\,  F^i(A) \mid i=0,1,2,\dots )$ of subgroups: that is, we have $F^0(A)=A$,   and  the inclusion $F^i(A)\supseteq F^{i+1}(A)$ holds for all integers $i\in \N_0$.
The filtered  abelian group $A$ is said to be \emph{separated} if its filtration is separated, i.e., we have $\bigcap F^i(M)=(0)$; it is \emph{complete} if the natural map $M\to \varprojlim M/F^i(M)$ is an isomorphism.

We now record some basic properties   of filtered abelian groups. (For general references, see \cite{HvO}, \cite{NvO}.)

\begin{enumerate}[label=(FilAb\arabic*) , font=\bfseries,  wide=0pt, itemsep=5pt]

\item \label{fab:val}  Let $A$ be a filtered abelian group with filtration $(F^i(A))$. 
The filtration $(F^i(A)))$ allows us to define the \emph{associated  graded abelian group} $\gr(A)$ given by
\[\gr(A):=\frac{F^0(A)}{F^1(A)}\oplus\frac{F^1(A)}{F^2(A)}\oplus\dots .\]
The filtration also gives rise to two maps: the \emph{valuation} map $v\colon A\to \Z\cup\{\infty\}$ and the \emph{principal part} map $\sigma \colon A\to \gr(A)$. These are    defined as follows.  If $a\in A$, then:
\begin{enumerate}[(i)]
\item  \label{d:val} the valuation of $a$ is given by $v(a):=\max\{\, i\in \N_0 \mid a\in F^i(A)\,\}$;
\item \label{d:pp} the principal part $\sigma(a)$ is the image of $a$ in $\gr(A)$ under the natural map 
\[F^{v(a)}(A)\to F^{v(a)}(A)/F^{v(a)+1}(A).\]
\end{enumerate}
By convention  $F^\infty(A):=\bigcap F^i(A)$. Thus  an element $a\in A$ has infinite valuation $v(a)=\infty$ if and only if the principal part $\sigma(a)=0$.

\item \label{fab:gag} Let $A=A_0\oplus A_1\oplus \dots $ be an $\N_0$-graded abelian group. Then $A$ is naturally filtered with filtration $F^i(A)=A_i\oplus A_{i+1}\oplus \dots $,  and we have $\gr(A)\cong A$ canonically. 

\item \label{fab:subgroup} Subgroups and quotients of a filtered abelian group inherit natural filtrations. Let $A$ be a filtered abelian group with filtration $(F^i(A))$ and 
let $B$ be a subgroup of $A$. Then we have induced filtrations on  $B$ and $A/B$ given by
\[F^i(B)=B\cap F^i(A)\quad\text{and}\quad F^i(A/B)=(F^i(A)+B)/B.
\] The natural map  $F^i(B)/F^{i+1}(B)\to F^i(A)/F^{i+1}(A)$ is an injection for all $i\in\N_0$, and  $\gr(B)$ is naturally a graded subgroup of $\gr(A)$. The valuation and principal part of an element in $B$ does not depend on whether we use the filtration $(F^i(B))$ or $(F^i(A))$.

\forceindent We now consider products. Let $A_1, \dots , A_n$ are filtered abelian groups. Then the induced filtration on   $A_1\times \dots \times  A_n$ is 
given by 
\begin{equation}\label{eq:filprod}
 F^i(A_1\times \dots \times  A_n)=F^i(A_1)\times \dots \times  F^i(A_n).
 \end{equation}
 We have the following equality for the induced valuation: if $(a_1,\dots , a_n)$ is an element of $A_1\times \dots \times  A_n$, then  
\begin{equation} \label{eq:valprod}v(a_1,\dots , a_n)=\min\{v(a_1),\dots , v(a_n)\}.\end{equation}

\item \label{fab:top} The abelian group $A$ is a topological group with its associated  filtration $(F^i(A))$ forming a basis of open neighbourhoods of $0$.  If the filtration is separated,  then the topology on $A$ is the metric topology induced by the valuation map.    If $A$ is complete, then  the natural map $A\to \varprojlim A/F^i(A)$ is an isomorphism of topological groups with each factor $ A/F^i(A)$ in the product having the discrete topology,  and $A$ is a complete metric space.  

\end{enumerate}

We now state two lemmas that are key tools in the construction of elements in complete filtered abelian groups.   The first of these (Lemma \ref{L:filab convergence})  shows how  to handle problems of convergence using the valuation; the proof is standard and omitted (see, for instance, \cite[Ch. I, \S 3.3]{HvO}).  The second (Lemma \ref{L:filabmain}) is typically used to lift structures from associated gradations.

\begin{lemma}\label{L:filab convergence} Let $A$ be a filtered abelian group,  and  let $v\colon A\to \Z\cup\{\infty\}$ be its valuation map. Then the following statements hold.
\begin{enumerate}[(1),leftmargin=*]
\item Suppose  $A$ is separated. Then a sequence $(x_n)$  in $A$ converges to $x\in A$ if and only $v(x_n-x)\to \infty$.
\item Suppose $A$ is complete. If $(x_n)$ is a sequence in $A$, then the series $x_1+x_2+\dots $ converges precisely when $v(x_n)\to \infty$ (i.e., when  $x_n\to 0$).
\end{enumerate}
\end{lemma}

\begin{lemma}\label{L:filabmain} Let $A$ and $B$ be filtered abelian groups, and let $\phi : A\to B$ be a continuous map. Assume the following two hypotheses.
\begin{enumerate}[(i)]
\item \label{fag1} $A$ is complete and $B$ is separated. 
\item \label{fag2}There exists a constant $K\in \Z$ with the following property: given an element $b\in B$, we can find an element $a\in A$ such that
\[ v(b)< v(b-\phi(a))\quad \text{and}\quad v(b)\leq v(a)+ K.\]
\end{enumerate}
Then $\phi : A\to B$ is surjective.
\end{lemma}

\begin{proof} Let $b\in B$. We construct  sequences $a_1,a_2,\dots $ (in $A$) and $b_0,b_1,b_2,\dots $ (in $B$) as follows. Set $b_0:=b$. Then, using hypothesis \ref{fag2}, choose an element $a_1\in A$ such that 
\[ \text{$v(b_0)< v(b_0-\phi(a_1))$ and $v(b_0)\leq v(a_1)+ K$,}\]
and set $b_1:=b_0-\phi(a_1)$. Then iterate. 

More formally, having constructed $a_0,a_1,\dots , a_k\in A$ and $b_1,\dots ,b_k\in B$, we use hypothesis \ref{fag2} to choose an element $a_{k+1}\in A$ and set $b_{k+1}=b_k-\phi(a_{k+1})$ so that 
\[
v(b_k)<  v(b_{k+1})\quad \text{and}\quad  v(b_k)\leq v(a_{k+1})+K.\]
Since  $b_k\to 0$ and $B$ is separated, the sequence $(\phi(a_1)+\dots +\phi(a_k))$ converges to $b$ by Lemma \ref{L:filab convergence}.  Also, as $a_k\to 0$ and $A$ is complete, the series $a_1+a_2+\dots $ converges in $A$, say to $a\in A$. We must then have $\phi(a)=b$ by the continuity of $\phi$, and this completes the proof of the lemma.
\end{proof}

\subsection{Filtered rings and filtered spaces}\label{ss:filring} By a \emph{filtered ring}, we mean a ring $R$ together with a decreasing filtration $(\,  F^i(R) \mid i=0,1,2,\dots )$ 
of ideals  satisfying the following conditions.
\begin{itemize}
\item The filtration is exhaustive: that is, we have $F^0(R)=R$.
\item  The inclusions $F^i(R)F^j(R)\subseteq F^{i+j}(R)$ hold for all $i,j\geq 0$.    
\end{itemize}
If  $R$ is a filtered ring, then a \emph{filtered $R$-space} is an $R$-space $M$ together with a decreasing filtration of $R$-subspaces $(\,  F^i(M) \mid i=0,1,2,\dots )$ such that $ F^0(M)=M$ and 
 $F^i(R)F^j(M)\subseteq F^{i+j}(M)$ for all $i,j\geq 0$.

Let  $R$ is a filtered ring.  Then the graded abelian group
\[\gr (R):= \frac{F^0(R)}{F^1(R)}\,\oplus \, \frac{F^1(R)}{F^2(R)}\,\oplus \, \dots .\] is a ring. We refer to $\gr(R)$ as the \emph{associated graded ring}. Furthermore, if $M$ is a filtered $R$-space then the associated graded abelian group
\[\gr (M):= \frac{F^0(M)}{F^1(M)}\,\oplus \, \frac{F^1(M)}{F^2(M)}\,\oplus \, \dots\]
is a graded $\gr(R)$-space.  The multiplicative structures in both cases are induced by the inclusions $F^i(R)F^j(R)\subseteq F^{i+j}(R)$ and 
$F^i(R)F^j(M)\subseteq F^{i+j}(M)$.

\begin{remark} The products $F^i(R)F^j(R)$ and  $F^i(R)F^j(M)$ are  additive groups (Definition \ref{D:fin span}); they are not assumed to be ideals or subspaces.\end{remark}

Filtered rings and filtered spaces are filtered abelian groups of course,  and we shall use the terminology introduced in Section \ref{ss:fag} without explicit reference. Note that  valuation maps and principal part maps  for filtered rings and filtered spaces do not preserve neither addition nor multiplication.  We record  the following specific case for later use, and  briefly discuss two examples of filtrations and associated gradations.

\begin{lemma}\label{L:val} Let $M$ be a filtered $R$-module.  Suppose  $r\in R$ and $m\in M$ satisfies $\sigma(r)\sigma(m)\neq 0$ in $\gr(M)$. Then $v(rm)=v(r)+v(m)$ and $\sigma(rm)=\sigma(r)\sigma(m)$.
\end{lemma}
 
\begin{proof} Set  $v(r)=i$ and $v(m)=j$ (both valuations are necessarily integers). Since
$\sigma(r)\sigma(m)=rm\mod{F^{i+j+1}}$ by definition,   the product $rm$ is an element of the set $F^{i+j}(M)\mysetminus F^{i+j+1}(M)$. Hence $v(rm)=i+j$ and $\sigma(rm)=\sigma(r)\sigma(m)$.\end{proof}

\begin{example}\label{E:poly} Let $R$ be a ring, let $M$ be an $R$-space, and let $X$ be a central indeterminate. Then $M[X]$ and $M[[X]]$ are filtered modules over the filtered rings $R[X]$ and $R[[X]]$, respectively. In all instances, the filtrations are induced by powers of $X$: if $A$ denotes one of $R[X]$, $R[[X]]$, $M[X]$  or $M[[X]]$, then $F^i(A)=AX^i$. The filtrations are separated in the polynomial cases (namely, $R[X]$ and $M[X]$), complete in  the formal power series cases (namely, $R[[X]]$ and $M[[X]]$), and we have canonical isomorphisms 
\[\gr(R[[X]])\cong R[X],\quad \gr(M[[X]])\cong M[X].\]
\end{example}

\begin{example}\label{E:filpoly} More generally,  let $M$ be a filtered  $R$-space and let $X$ be a central indeterminate. 
Then $M[X]$ and $M[[X]]$ are  filtered modules over the filtered rings $R[X]$ and $R[[X]]$, respectively, with canonical  filtrations  induced by the valuation (on $R$ and $M$, respectively) and the indeterminate $X$ having valuation $1$.

Explicitly, let $A$  denote the ring $R$ or the space $M$.  Then the valuation  and filtration on $A[[X]]$ are given by
 \begin{align}\label{eq:valfilpoly}
 \begin{split}
  v\big(a_0+a_1X+\dots\big)&=\min\big\{\, v(a_0), v(a_1)+1,\dots \,\big\}, \\
  F^n(A[[X]])& =\{\, f\in A[[X]]\; |\; v(f)\geq n\,\} , \end{split}\end{align}
and  $A[X]$ inherits these as a subspace of $A[[X]]$.
The corresponding associated graded objects $\gr(A[[X]])$ and $\gr(A[X])$ are bigraded over $\N_0\times \N_0$ with $(i,j)$-th homogeneous component $({F^i(A)}/{F^{i+1}(A)})X^j$, and  we have canonical identifications
\begin{equation} \label{eq:monfilpoly}
\gr(A[[X]])\cong \gr(A[X])\cong\gr(A)[X]=
\bigoplus_{(i,j)\in\N_0\times\N_0}\frac{F^i(A)}{F^{i+1}(A)}X^j .
\end{equation}
(Note that Example \ref{E:poly} is the case when $F^i(A)=(0)$ for $i\geq 1$.)
\end{example}
As indicated in the introduction, a recurring theme of this article is to illustrate how properties of filtered rings can be deduced from their associated graded rings; the following proposition is an example of this motif.

\begin{proposition}\label{P:fsn} Let $R$ be a complete filtered ring and let $M$ be a separated filtered $R$-space. If $\gr(M)$ is a strongly Noetherian $\gr(R)$-space then $M$ is a strongly Noetherian $R$-space.
\end{proposition}

For an application, let $R$ be a ring and let $M$ be a strongly Noetherian $R$-space. Then $\gr(M[[X]])\cong M[X]$ is a strongly Noetherian $\gr(R[[X]])$-space by Proposition \ref{P:hbt}\ref{hbt2}. Using completeness of $R[[X]]$ and $M[[X]]$, and Proposition \ref{P:fsn}, we deduce the following corollary.

\begin{corollary}\label{C:fps} Let $R$ be a ring and let $M$ be a strongly Noetherian $R$-space. Then $M[[X]]$ is a strongly Noetherian $R[[X]]$-space.
\end{corollary}

We now turn to the proof of  Proposition \ref{P:fsn}. Our  argument relies on being able to lift spanning sets from the associated graded space; to this end, we have the following lemma.

\begin{lemma}\label{L:lifting generators} Let $R$ be a complete filtered ring and let $N$ be a separated filtered $R$-space.  Suppose we are given finitely many elements $y_1,\dots , y_n\in N$ such that 
\[ \gr(M)=\gr(R)\sigma(y_1)+\dots +\gr(R)\sigma(y_n). \]Then  $N=Ry_1+\dots +Ry_n$.\end{lemma}

\begin{proof} Without loss of generality, we may assume that $y_1,\dots ,y_n$ are all non-zero. As we saw in \ref{fab:subgroup}, the $n$-fold product  $R^n$  is a complete filtered abelian group with its induced filtration, and  the valuation of an element $(r_1,\dots ,r_n)\in R^n$ is given by   $v(r_1,\dots ,r_n)=\min\{r_1,\dots , r_n\}$.

Now consider the map $\phi: R^n\to N$ given by 
\[\phi(r_1,\dots, r_n)= r_1y_1+\dots +r_ny_n.\] The map $\phi$ is clearly a continuous additive group homomorphism from a complete group to a separated group. The lemma is then equivalent to surjectivity of $\phi$. To this end, let us consider
 a non-zero element $y\in N$. By expressing the principal part  $\sigma(y)$ as a $\gr(R)$ linear combination of $\sigma(y_1),\dots, \sigma(y_n)$ and then separating out into homogeneous parts, we see that we can find an element $(r_1,\dots , r_{n})\in R^n$ such that
\begin{itemize}
\item $\sigma(y)=\sigma(r_{1})\sigma(y_1)+\dots +\sigma(r_{n,})\sigma(y_n)$, and
\item for each $i=1,\dots , n$, either $r_{i}=0$ or $v(r_{i})+v(y_i)= v(y)$.
\end{itemize}
Setting  $d:=\max\{v(y_1),\dots , v(y_n)\}$, we then obtain the following inequalities:
\[
 v(y) <v\left(y-\phi(r_1,\dots ,r_n)\right)\quad \text{and}\quad 
v(y)\leq  \min\{v(r_{1}),\dots ,v(r_{n})\}+d.
\]
Hence, by Lemma \ref{L:filabmain}, the map $\phi:R^n\to N$ is surjective, and this completes the proof.\end{proof}

\begin{proof}[Proof of Proposition \ref{P:fsn}]  Let $N$ be an $R$-subspace of $M$. Then $N$ is a separated filtered $R$-space under the induced filtration $F^i(N)=N\cap F^i(M)$ and $\gr(N)$ is naturally a $\gr(R)$-subspace of $\gr(M)$. Since $\gr(M)$ is strongly Noetherian, we can find finitely many elements $y_1,\dots, y_n$ in $ N$ such that 
\[
\gr(N)=\gr(R)\sigma(y_1)+\dots +\gr(R)(y_n).\]
  By Lemma \ref{L:lifting generators}, we have $N=Ry_1+\dots +Ry_n$. Hence $M$ is strongly Noetherian.
\end{proof}

\section{The Artin--Rees lemma for complete local-filtered rings}\label{S:AR}

We will now prove assertions \ref{A1} to \ref{A4} in the setting of filtered spaces.  As noted in Section \ref{S:introduction}, the key idea is to consider rings with good filtrations and derive the sought results from associated gradations. Motivated by Zariskian filtrations in the associative setting, we  consider filtered rings with strongly Noetherian Rees rings in Section \ref{ss:rr}. We then specialise to filtered rings with affine associated graded rings in  the remaining sections;  a link with local rings is established in Section \ref{ss:affine}, and we prove the Artin--Rees lemma for filtered  spaces in Section \ref{ss:AR}.

\subsection{The Rees ring} \label{ss:rr} Let $R$ be a filtered ring with filtration $(F^i(R))$. As in the associative setting, we define the Rees ring $\rr(R)$ of $R$ to be  the subring of $R[X]$ given by
\begin{equation}\label{eq:reesring}\mathbf{R}(R)=R\oplus F^1(R)X\oplus F^2(R)X^2\oplus \dots.\end{equation}
If $M$ is a filtered $R$-space, then the associated Rees space $\rr(M)\subseteq M[X]$ is defined in a similar fashion: 
\begin{equation}\label{eq:reesspace}\rr(M):=F^0(M)\oplus F^1(M)X\oplus F^2(M)X^2\oplus \dots .\end{equation}It is easily checked that $\rr(M)$ is, indeed, an $\rr(R)$-space.

Let $M$ be a filtered $R$-space, and let
 $A$ denote the ring $R$ or the space $M$. Recall that $A[X]$ has a natural filtration with associated valuation extending the valuation on $A$  and  the indeterminate $X$ having valuation $1$ (see Example \ref{E:filpoly}), and $\rr(A)\subseteq A[X]$ inherits the filtration. Explicitly, the  induced filtration $(F^\ast\rr(A))$ on  $\rr(A)$ is given by
\begin{equation}\label{eq:reesfil}
F^n(\rr(A)) =\big\{\, f\in \rr(A)\; |\; v(f)\geq n\,\big\}=\bigoplus_{\substack{ j\leq i \\ n\leq i+j}} F^i(A)X^j\end{equation}
with valuation $v$ given by the equality \eqref{eq:valfilpoly}.
The associated graded object $\gr(\rr(A))$ is bigraded with $(i,j)$-th homogeneous component $({F^i(A)}/{F^{i+1}(A)})X^j$ when $ i\geq j\geq 0$, and $0$ otherwise. Combining with \ref{fab:gag}, we obtain canonical identifications
\begin{equation} \label{eq:monomials}
\gr(\rr(A))\cong 
\bigoplus_{i=0}^\infty\left(\bigoplus_{j=i}^\infty\frac{F^j(A)}{F^{j+1}(A)}X^i\right)\cong \rr(\gr(A)).
\end{equation}
Thus  we have a natural identification of  $\gr(\rr(A))$ as a subset of  $\gr(A[X])$, and a straightforward  check shows that $\gr(\rr(M))$ is a filtered $\gr(\rr(R))$-space.

  We can  now state  a proposition on
the strong Noetherian property for  Rees rings and spaces. 

\begin{proposition}\label{P:RR1} Let $R$ be  a complete filtered ring and let $M$ be a separated filtered $R$-space. If  $\gr(\rr(M))$ is strongly Noetherian, then $\rr(M)$ is strongly Noetherian.  
\end{proposition}

We begin the proof of Proposition \ref{P:RR1} with a definition. Let $M$ be a filtered $R$-space, and let
 $A$ denote $M$ or the $R$-space $R$. We define the \emph{leading monomial part} map $\varphi : \rr(A)\to \gr(\rr(A))$ as follows: if $\alpha\in \rr(A)$, then
\begin{equation}\label{eq:leadmon}
\varphi(\alpha)=\begin{cases}
 \sigma(a_n)X^n, & \text{if  $\alpha=a_0+\dots +a_nX^n$ has degree $n$};\\
 0, &\text{if  $\alpha=0$.}\end{cases}\end{equation}
  As in the case of the principal part map $\sigma: A\to \gr(A)$, the leading monomial part map preserves  multiplication only in specific cases, one of which is covered by the following lemma.

\begin{lemma} \label{L:lmp} Let $\alpha\in \rr(R)$ and $\beta\in \rr(M)$. If  $\varphi(\alpha)\varphi(\beta)\neq 0$, then $\varphi(\alpha\beta)=\varphi(\alpha)\varphi(\beta)$. \end{lemma}
\begin{proof} Immediate from Lemma \ref{L:val}.
\end{proof}

We now return to the proof of Proposition \ref{P:RR1}.  Let $\mathcal{N}$ be an $\rr(R)$-subspace $\rr(M)$, and let $\overline{\varphi(\mathcal{N})}$ be the $\gr(\rr(R))$-subspace of $\gr(\rr(M))$ generated by $\varphi(\mathcal{N})$. Then by Lemma \ref{L:lmp} above, we see that 
\[\overline{\varphi(\mathcal{N})}=\gr(\rr(R))\varphi(\mathcal{N})=\left\{\,\text{finite sums $\sum\varphi(\eta)\, \Big|\, \eta\in \mathcal{N}$}\,\right\}.\]
Hence,  since $\gr(\rr(M))$ is strongly Noetherian, we can find finitely many non-zero elements $\eta_1,\dots , \eta_n$ in $\mathcal{N}$ such that 
\[ \overline{\varphi(\mathcal{N})}=\gr(\rr(R))\varphi(\eta_1)+\dots +\gr(\rr(R))\varphi(\eta_n).\]
Our  aim then is to  show that $\mathcal{N}=\rr(R)\eta_1+\dots+\rr(R)\eta_n$.  

Note that we cannot hope to use the method of Lemma \ref{L:lifting generators} directly as the Rees ring $\rr(R)$ may not be complete. We  get around this issue by working with elements of a fixed degree as follows. 

For each non-negative integer $j\in \N_0$, we define---as  in the displayed description \eqref{eq:LT}---the  $R$-space $\mathcal{N}(j)\subseteq M$ by  
\[ \mathcal{N}(j):=\{\, m\in M\, \vert\, (mX^j+\text{lower degree terms}) \in \mathcal{N}\,\}.\]
It is clear from the definition of $\rr(M)$ that $\mathcal{N}(j)\subseteq F^j(M)$.
Also, for each index $i=1,\dots ,n$, let $d_i\in \N_0$ be the degree of $\eta_i$ and let $a_i\in M$ be the leading coefficient of $\eta_i$. Thus 
\begin{equation}\label{eq:etai}
\eta_i=a_iX^{d_i}+\text{lower degree terms}, \quad\text{and}\quad 
\varphi(\eta_i)=\sigma(a_i)X^{d_i}.
\end{equation}

We then have the following claim.

\begin{claim}\label{claim:RR1} Let $\eta_i=a_iX^{d_i}+\text{lower degree terms}$, $i=1,\dots , n$, be as above. If $j$ is a negative integer, then we  set $F^j(R)=0$. Fix a non-negative integer  $d\in\N_0$, and define the map
\[ \Phi :  F^{d-d_1}(R)\times \dots \times F^{d-d_n}(R)\to \mathcal{N}(d)\]
 by $\Phi(r_1,\dots ,r_n):=r_1a_1+\dots +r_na_n$. Then $\Phi$ is surjective.
\end{claim}

\begin{proof}[Proof of Claim \ref{claim:RR1}] We give $F^{d-d_1}(R)\times \dots \times F^{d-d_n}(R)$ and $\mathcal{N}(d)$ their natural filtrations (inherited from $R^n$ and $M$, respectively; see the displayed equalities \eqref{eq:filprod} and \eqref{eq:valprod}). Thus an element $(r_1,\dots ,r_n)\in F^{d-d_1}(R)\times \dots \times F^{d-d_n}(R)$ has valuation
\[ v(r_1,\dots, r_n)=\min\{v(r_1), \dots ,v(r_n)\}.\] It is clear that $\Phi$ is  a well defined, continuous, additive homomorphism from a complete filtered abelian group to a separated abelian group.

Consider a non-zero element $0\neq m\in \mathcal{N}(d)$. We choose an element $\alpha\in \mathcal{N}$ of degree $d$ and leading coefficient $m$ so that
\[
\alpha=mX^d+\text{lower degree terms}\quad\text{and}\quad \varphi(\alpha)=\sigma(m)X^d.\] 
Since $\overline{\varphi(\mathcal{N})}=\gr(\rr(R))\varphi(\eta_1)+\dots +\gr(\rr(R))\varphi(\eta_n)$, we can write 
\[\sigma(m)X^d=\mathfrak{r}_1 \sigma(a_1)X^{d_1}+ \dots +\mathfrak{r}_n\sigma(a_n)X^{d_n}\] for some $\mathfrak{r}_1,\dots ,\mathfrak{r}_n$ in $\gr(\rr(R))$. Separating  out each $\mathfrak{r}_i$ into homogeneous parts (using the monomial decomposition of $\gr(\rr(R))$ given by equality \eqref{eq:leadmon}), and using Lemma \ref{L:val}, we see that we  can find an element 
\[
(r_{1},\dots , r_{n})\in F^{d-d_1}(R)\times \dots \times F^{d-d_n}(R)\]
 such that
\begin{itemize}
\item $\sigma(m)=\sigma(r_{1})\sigma(a_1)+\dots +\sigma(r_{n})\sigma(a_n)$, and
\item for each $i=1,\dots , n$, the component $r_{i}\in F^{d-d_i}(R)$ is either equal to $0$, or else we have $v(r_{i})+v(a_i)= v(m)$.
\end{itemize}
The two properties of $(r_1,\dots ,r_n)$ listed immediately above show that we have inequalities 
\begin{align*} v(m)& <v \big(m-(r_1a_1+\dots +r_na_n)\big) \quad\text{and}\\
v(m)&\leq \min\{ v(r_1),\dots ,v(r_n)\}+ \max\{v(a_1),\dots ,v(a_n)\}.\end{align*}
The surjectivity of $\Phi$ now follows from Lemma \ref{L:filabmain}.
\end{proof}

\begin{proof}[Completion of proof of Proposition \ref{P:RR1}] We will now show that 
$\mathcal{N}$ is spanned by the elements $\eta_1, \dots, \eta_n$.  
 For a contradiction, let us suppose that $\mathcal{N}$ strictly contains the abelian group $\rr(R)\eta_1+\dots+\rr(R)\eta_n$. We can then pick an element  
\[\eta\in \mathcal{N}\mysetminus (\rr(R)\eta_1+\dots+\rr(R)\eta_n)\]
 of minimal degree, say $d$. Let $m\in M$ be the leading coefficient of $\eta$; so 
\[ \eta=mX^d+\text{lower degree terms}, \quad \text{and}\quad \varphi(\eta)=\sigma(m)X^d.\]
By Claim \ref{claim:RR1}, we can find an element $(r_1,\dots ,r_n)\in  F^{d-d_1}(R)\times \dots \times F^{d-d_n}(R)$ such that
$r_1a_1+\dots +r_na_n=m$. 

Now consider $\widetilde{\eta}:= r_1X^{d-d_1}\eta_1+\dots +r_nX^{d-d_n}\eta_n$. Then  $\widetilde{\eta}$ has leading term $mX^d$, and is an element of 
$\rr(R)\eta_1+\dots+\rr(R)\eta_n$. But then 
\[\eta-\widetilde{\eta}\notin \big(\rr(R)\eta_1+\dots+\rr(R)\eta_n\big)\] and the degree of $\eta-\widetilde{\eta}$ is strictly less than $d$, contradicting our choice of $\eta$.\end{proof}

\subsection{Complete local-filtered rings} \label{ss:affine}

From here on, we will be primarily interested in the following  class of filtered rings.

\begin{definition}\label{D:CLF} We say that a filtered ring  $R$ is a
 \emph{complete local-filtered ring}, or that $R$ has a complete local-filtration,  if its filtration $(F^i(R))$ is complete and satisfies the following properties.
\begin{enumerate}[label=(C\arabic*)]
\item \label{clf1}  The quotient  $R/F^1(R)$ is a field, which we refer to as the \emph{residue field} of $R$.
\item \label{clf2} The associated graded ring $\gr(R)$ is  a finitely generated  commutative $R/F^1(R)$-algebra and generated in degree $1$.
\end{enumerate}
\end{definition}

We then have the following structure theorem for complete local-filtered rings.

\begin{theorem}\label{T:CLF} Let $R$ be a complete local-filtered ring. Then the following statements hold.
\begin{enumerate}[(1)]
\item \label{affine1} $R$ is a local ring with maximal ideal $F^1(R)$.
\item \label{affine2}$R$ is strongly left noetherian.
\item \label{affine3}  For all integers $i,j\in\N_0$, we have $F^i(R)F^j(R)=F^{i+j}(R)$.
\item \label{affine4} The Rees ring $\rr(R)$ is strongly left Noetherian.
\end{enumerate}
\end{theorem}

 Before proceeding to the proof of Theorem \ref{T:CLF},  we record the following note on  parts \ref{affine1} and \ref{affine3}  of Theorem \ref{T:CLF} (thereby  justifying our terminology in Definition \ref{D:CLF}).  Here and subsequently, we write `\emph{$(R,\m)$ is a complete local-filtered (resp. local) ring}' to express the condition that $R$ is a complete local-filtered (resp. local) ring with maximal ideal $\m$.
 
\begin{note}\label{N:CLF} Let $(R,\m)$ be a local ring, and let $n$ be a positive integer. For each choice of placing parenthesis in an $n$-letter word that allows evaluation, we can derive a corresponding $n$-fold product of $\m$. In general, different choices of inserting parenthesis will lead to different additive subgroups of $R$. 
 
 Suppose now that the local ring $(R,\m)$ has a complete local-filtration.  Then part \ref{affine1} and part \ref{affine3} of Theorem \ref{T:CLF} imply that all choices of inserting parenthesis will give the same additive subgroup and we obtain a uniquely defined $n$-fold product $\m^n$. For definiteness, recursively set $\m^{n+1}:=\m(\m^n)$ for integers  $n\in \N_0$ with $\m^0:=R$. Then each $\m^n$ is an ideal of $R$ and
 $(\m^n\mid n=0,1,2,\dots)$ is \emph{the} complete local-filtration. Furthermore, the equality $\m^i\m^j=\m^{i+j}$ holds for all non-negative integers $i$ and $j$.\end{note}

\begin{proof}[Proof of Theorem \ref{T:CLF}] We denote the residue field $F^0(R)/F^1(R)$ by $\pmb{k}$ throughout this proof.
\begin{enumerate}[label=(Part \arabic*) ,   wide=0pt , itemsep=3pt]
\item To show that $R$ is a local ring, consider an element $a\in R\mysetminus F^1(R)$. Now multiplication by $a$ induces an isomorphism on each homogeneous component of $\gr(R)$.  The continuous additive map $\phi :R\to R$ given by $\phi(x)=xa$ therefore satisfies the hypothesis of Lemma \ref{L:filabmain}: that is, given $y\in R$ we can find $x\in R$ such that $x\in R$ such that $v(y)<v(y-\phi(x))$ and $v(y)=v(x)$. Hence, by Lemma \ref{L:filabmain}, the map $\phi :R\to R$ is surjective and $a$ has a left inverse. It follows that every proper  left ideal of $R$ is contained in $F^1(R)$, and therefore $R$ is a local ring with maximal ideal $F^1(R)$.

\item This follows from Proposition \ref{P:fsn}.

\item  Let $i, j$ be non-negative integers. Fix elements $x_1,\dots ,x_n$ in  $F^j(R)$ such that their principal parts $\sigma(x_1),\dots , \sigma(x_n)$ form a basis of  $F^j(R)/F^{j+1}(R)$ as a $\pmb{k}$-vector space, and define a map
\[ \Phi: (F^i(R))^{ n}\to F^{i+j}(R)\quad\text{by $\Phi(r_1,\dots ,r_n):=r_1x_1+\dots+r_nx_n$.}\] 
The filtration on $R$ induces filtrations on $ (F^i(R))^{ n}$ and $F^{i+j}(R)$ (see \ref{fab:subgroup}), and  the map $\Phi$ is a continuous group homomorphism from the complete abelian group $F^i(R)^n$ to the complete abelian group $F^{i+j}(R)$. Since $\gr(R)$ is a commutative $\pmb{k}$-algebra of finite type and  generated in degree $1$, the  multiplication map
\[\frac{F^k(R)}{F^{k+1}(R)}\times \frac{F^j(R)}{F^{j+1}(R)}\to \frac{F^{k+j}(R)}{F^{k+j+1}(R)}\] is surjective. Thus if $b\in {F^{k+j}(R)}\mysetminus {F^{k+j+1}(R)}$ we can find $(a_1,\dots ,a_n)\in (F^k(R))^n$ such that 
\[b\equiv \Phi(a_1,\dots ,a_n)\mod{F^{k+j+1}(R)}.\]
Consequently, we have $v(b)< v\left(b-\Phi(a_1,\dots ,a_n)\right)$ and $v(b)= v(a_1,\dots ,a_n)+j$.
Hence, by Lemma \ref{L:filabmain}, the map $\Phi$ is surjective. Therefore $F^i(R)F^j(R)\supseteq F^{i+j}(R)$, and the claim follows.

\item Suppose $\gr(R)$ is generated as a $\pmb{k}$-algebra by the elements $x_1,\dots , x_n$ in $F^1(R)/F^2(R)$. From the displayed  equality \eqref{eq:monomials}, we see that $\gr(\rr(R))$ is the $\gr(R)$-subalgebra of $\gr(R)[X]$ generated by $x_1X,\dots , x_nX$. Hence $\gr(\rr(R))$ is strongly Noetherian, and the result follows from Proposition \ref{P:RR1}.\qedhere
\end{enumerate}
\end{proof}

\subsection{The Artin-Rees lemma }\label{ss:AR}

Just as the in the case modules over local rings, the following  version of the  Artin--Rees lemma establishes when the topology on filtered spaces over a complete local-filtered ring is induced by powers of the  maximal ideal.

\begin{theorem}[Artin--Rees lemma]
\label{T:AR3} Let $R$ be a complete local-filtered ring, and let $M$ be a separated filtered $R$-space. Assume that  the associated space 
$\gr(M)$ is a finitely generated \emph{$\gr(R)$-module}. 
Then there exists an integer $D\in \N_0$ such that 
\[F^{n+d}(M)=F^n(R)F^d(M)\]
for all integers $n\in \N_0$ and $d\geq D$.
\end{theorem}

For the proof, we need the following lemma which constructs spanning sets that preserve each step of the filtration. (The integer $D$ in Theorem \ref{T:AR3} is derived from the valuations of elements in these spanning sets.) 

\begin{lemma}\label{L:AR1} Let $R$ be a complete filtered ring and let $M$ be a separated filtered $R$-space. Suppose the associated Rees space $\rr(M)$ is a strongly Noetherian $\rr(R)$-space. We can then find finitely many non-negative integers $d_1,\dots , d_k$ in $ \N_0$ and elements $m_i\in F^{d_i}(M)$ for $i=1,\dots , k$ such that the following holds: for all $n\in \N_0$, we have
\[ F^n(M)=F^{n-d_1}(R)m_1+\dots +F^{n-d_k}(R)m_k\]
where we take $F^i(R):=(0)$ when $i<0$.
\end{lemma}

\begin{proof} Replacing each element of a finite spanning set with its monomial parts, we see that we can find finitely many non-negative integers 
 $d_1,\dots , d_k$ in $ \N_0$ and monomials $m_1X^{d_1},\dots , m_kX^{d_k}$ in $\rr(M)$ such that 
 \[\rr(M)=\rr(R)m_1X^{d_1}+\dots +\rr(R)m_k X^{d_k}.\]
Then each $m_i\in F^{d_i}(M)$ and $F^n(M)=F^{n-d_1}(R)m_1+\dots +F^{n-d_k}(R)m_k$.\end{proof}

\begin{proof}[Proof of Theorem \ref{T:AR3}] We view  $\gr(M)$ as a filtered $\gr(R)$-module with their natural filtrations (induced from the gradings; see \ref{fab:gag}). Since $\gr(R)$ is affine and generated in degree $1$, it follows that $\rr(\gr(M))$ is a Noetherian $\rr(\gr(R))$-module. Using the canonical isomrphisms
\[ \gr(\rr(R))\cong \rr(\gr(R)),\quad \gr(\rr(M))\cong\rr(\gr(M))\]
(see the displayed equalities \eqref{eq:leadmon}) and Proposition \ref{P:RR1}, we deduce that $\rr(M)$ is a strongly Noetherian $\rr(R)$-space.

By Lemma \ref{L:AR1}, we can then find $k$ non-negative integers $d_1,\dots , d_k$ in $\N_0$ and elements $m_i\in F^{d_i}(M)$ for each  $i=1,\dots , k$ such that 
\begin{equation}\label{eq:oneoff}
F^n(M)=F^{n-d_1}(R)m_1+\dots +F^{n-d_k}(R)m_k\end{equation}
 for all $n\in \N_0$. Set $D:=\max\{d_1,\dots ,d_k\}$, and let $d\geq D$. Then, using Theorem \ref{T:CLF}\ref{affine4}, we obtain
\begin{equation} \label{eq:ar3proof} F^{n+d}(M)=\left(F^n(R)F^{d-d_1}(R)\right)m_1+\dots +\left(F^n(R)F^{d-d_k}(R)\right)m_k\end{equation} 
for all $n\in \N_0$.

For each $i\in \{1,\dots , k\}$, let  $n_i$ be the dimension of $F^{d-d_i}(R)/F^{d-d_i+1}(R)$  as a vector space over the residue field of $R$. Choose 
elements $x(i,1),\dots , x(i,n_i)$ in $F^{d-d_i}(R)$ 
such that the principal parts 
$\sigma(x(i,1)),\dots , \sigma(x(i,n_i))$ form a basis of $F^{d-d_i}(R)/F^{d-d_i+1}(R)$. 

Now let $n\in\N_0$ and let $m\in F^{n+d}(M)\mysetminus F^{n+d+1}(M)$. Using the displayed equality \eqref{eq:ar3proof},  we can find elements $r(i,j)\in F^n(R)$, where $i\in \{1,\dots ,k\}$ and $j\in \{1,\dots ,n_i\}$ for each $i$,
such that 
\[\sum_{i=1}^k \Big(\sum_{j=1}^{n_i} r(i,j)(x(i,j)\Big)m_i\equiv m\pmod{F^{n+d+1}(M)}.\]
The assumption that $\gr(M)$ is a $\gr(R)$-module then implies that 
\[\sum_{i=1}^k \sum_{j=1}^{n_i} r(i,j)((x(i,j)m_i)\equiv m\pmod{F^{n+d+1}(M)}.\]
It follows that for a fixed $N\in\N_0$, the map 
\[(F^N(R))^{n_1}\times \dots \times (F^N(R))^{n_k}\to F^{N+d}(M)\]
given by 
\[ \big(r(1,1),\dots , r(1,n_1), \dots , r(k,1),\dots ,r(k,n_k)\big)\to \sum_{i=1}^k \sum_{j=1}^{n_i} r(i,j)((x(i,j)m_i)\]
satisfies the hypotheses of Lemma \ref{L:filabmain}, and therefore is surjective.
Since 
\[x(i,j)m_i\in F^{d-d_i}(R)F^{d_i}(M)\subseteq F^d(M),\] we obtain
$F^{N+d}(M)\subseteq F^N(R)(F^d(M)$. This completes the proof of Theorem \ref{T:AR3}.\end{proof}

As an application of Theorem \ref{T:AR3}, we now discuss the extent to which  the filtration on  spaces with cyclic associated gradation  is determined by the maximal ideal. The example below will be particularly relevant in Section \ref{ss:torsion} when we discuss torsion elements.

\begin{example}[Cyclic spaces] \label{E:CS} 
Let $(R,\m)$ be a complete filtered ring,  and let $M$ be a separated filtered $R$-space. We assume that $\gr(M)$ is a cyclic $\gr(R)$-module, and fix an element $x\in M$  whose principal part generates $\gr(M)$. Then, by Lemma \ref{L:lifting generators}, the element $x$ spans $M$: that is, we have $M=Rx$. By re-indexing the filtration on $M$ if necessary, we can assume that $x$ has valuation $0$. We will now show that the following statements hold.
\begin{enumerate}[label=(CS\arabic*) , font=\bfseries, ]
\item \label{CS1} The filtration on $M$ is given by 
$F^{n}(M)=\m^nx$ for all integers $n\in \N_0$.
\item \label{CS2} For all integers $i,j\in\N_0$, we have $\m^i(\m^jx)=\m^{i+j}x$.   \end{enumerate}

\begin{proof} For part \ref{CS1}, we need to check that the map $\phi:F^n(R)\to F^{n}(M)$ given by $\phi(r)= rx$ is surjective. Now  $\gr(M)=\gr(R)\sigma(x)$ where $\sigma(x)\in M/F^{1}(M)$ is the principal part of $x$. Therefore, if $j\geq n$, then the map
\[F^j(R)\xrightarrow{r\to rx+ F^{j+1}(M)} F^{j}(M)/F^{j+1}(M)\] is surjective. The surjectivity of $\phi$ now follows from Lemma \ref{L:filabmain}.

For part \ref{CS2}, comparing $F^n(M)=F^n(R)x$ against the displayed relation \eqref{eq:oneoff} shows that Theorem \ref{T:AR3} applies with $D=0$. Thus $F^{i+j}(M)=F^i(R)F^j(M)$ for all integers $i,j\in \N_0$.
\end{proof}

\end{example}

\section{Dimensions and sizes of permissible spaces}\label{S:DPS}

As we will now be  considering consequences of the Artin--Rees lemma, let us denominate the  class of spaces satisfying the hypotheses of Theorem \ref{T:AR3}.

\begin{definition} \label{D:permissible} Let $R$ be a  complete local-filtered ring.
\begin{enumerate}[(1)]
\item A filtered space $M$ is said to be  \emph{permissible} if  it is separated and the associated graded space
$\gr(M)$ is a finitely generated \emph{$\gr(R)$-module}. If $M$ is permissible, then we call the Krull dimension of $\gr(M)$ the \emph{associated dimension} of $M$ and denote it by $\dim(M)$.
\item We say that a filtration is a \emph{permissible filtration} if it is the filtration of a permissible $R$-space.
\end{enumerate}
 \end{definition}

Let $R$ be a complete local-filtered ring with maximal ideal $\m$.  We will now show that   the topology and the associated dimension of a permissible $R$-space are independent of the (choice of) permissible filtration. When $R$ has finite residue field, the independence of filtration allows us to determine the cardinality of $M/\m^nM$ asymptotically for any permissible $R$-space $M$. Finally, in  Section \ref{ss:torsion}, we relate associated dimensions of permissible spaces to torsion.

\subsection{Independence of permissible filtrations}\label{ss:dim}

\begin{proposition}\label{P:permissibletop} Let $(R,\m)$ be a complete local-filtered ring, and let $M$ be a permissible filtered $R$-space. Then, the topology on  $M$ is induced by the filtration $(\m^iM\mid i=0,1,\dots)$ of subgroups.
\end{proposition}

\begin{proof} By Theorem \ref{T:AR3}, we can find a non-negative integer $d$ such that $F^{n+d}(M)=F^n(R)F^d(M)$ for all $n\in\N_0$. Consequently, we  have 
\begin{equation} \label{eq:pertop} F^{n+d}(R)F^0(M) \subseteq F^{n+d}(M)=F^n(R)F^d(M)\subseteq F^n(R) F^0(M)\subseteq F^n(M)\end{equation} for all $n\in\N_0$. Since the topology on each quotient $M/F^i(M)$ is discrete, the subgroup $F^i(R)M$ is open. The proposition then follows from Theorem \ref{T:CLF} and the inclusions \eqref{eq:pertop} above. \end{proof}

\begin{proposition} \label{P:kdim} Let $R$ be a  complete local-filtered ring,  and let  $M$ be an $R$-space. If $(G^i(M)\mid i=0,1,\dots)$ and $(H^i(M)\mid  i=0,1,\dots)$ are two permissible filtrations on $M$, then the $\gr(R)$-modules 
\[ \frac{G^0(M)}{G^1(M)}\oplus\frac{G^1(M)}{G^2(M)}\oplus \dots\quad \text{and}\quad \frac{H^0(M)}{H^1(M)}\oplus\frac{H^1(M)}{H^2(M)}\oplus \dots\]
have the same Krull dimension. In other words, the dimension of a permissible space is independent of the choice of permissible filtration.\end{proposition}

 The justification makes use of the Hilbert--Samuel polynomial of the associated graded module, and we recall this now.  Let $R$ be complete local-filtered ring, and let $M$ be an $R$-space. We write  $\ell(M)$ for the  length of the $R$-space $M$. Thus $\ell(M)=n$ if there is a saturated chain
\[ 0=M_0\subsetneq M_1\subsetneq \dots \subsetneq M_n=M\]
of $R$-subspaces; if no such chain exists, then $\ell(M)=\infty$.  The length of an $R$-space is well defined since, as indicated in Section \ref{ss:gen},  $R$-spaces are  abelian groups with operators in the set $R$ (see \cite[Ch. 3, \S 3.3]{jacobson2}).

Let $\pmb{k}$ denote the residue field of $R$, and  let  $M$ be a permissible $R$-space. Then, for any integer $n\in\N_0$,  we have
\[
\ell\big(M/F^n(M)\big)
=\sum_{k=1}^n\dim_{\;\pmb{k}}F^{k-1}(M)/F^k(M).\]If  $\mathcal{P}(x)$ is the Hilbert--Samuel polynomial of the $\gr(R)$-module $\gr(M)$, then the degree of $\mathcal{P}$ is the Krull dimension of $\gr(M)$ and $\ell\big(M/F^n(M)\big)=\mathcal{P}(n)$ for all but finitely many $n\in \N_0$. (See \S 13 of \cite{mats}.) We will often refer to $\mathcal{P}(x)$ as the Hilbert--Samuel polynomial of the (permissible) filtration $(F^*(M))$.

\begin{note}\label{N:induced} Let $0\to L\to M\to N\to 0$ be a short exact sequence of $R$-spaces. We view $L$ as a subspace and $N$ as the quotient $M/L$. Let $(F^*(M))$ be a permissible filtration on $M$. Then, with the induced filtrations (see \ref{fab:subgroup}) on  $L$ and $N$, the sequence $0\to\gr(L)\to \gr(M)\to\gr(N)\to 0$ is an exact sequence of graded spaces: that is, we have
\begin{equation}\label{eq:sesexample}
0\to \frac{F^i(L)}{F^{i+1}(L)}\to \frac{F^i(M)}{F^{i+1}(M)}\to \frac{F^i(N)}{F^{i+1}(N)}\to 0\end{equation}
for all $i\in \N_0$. Thus the induced filtration on $L$ is permissible. If $L$ is topologically closed, then the quotient $M/L$ is also permissible with the induced filtration.
\end{note}

\begin{proof}[Proof of Proposition \ref{P:kdim}] We begin by setting out some notation. If $L$ is a permissible $R$-space with filtration $(\Lambda^*(L))$, then we denote its Krull dimension  by  $\dim(\Lambda^*(L))$.

Consider the filtrations on $M$  and $M\times M$ given by  
\[
F^i(M)= G^i(M)\cap H^i(M) \quad \text{and}\quad F^i(M\times M)=G^i(M)\times H^i(M),\]
respectively. Note that the product  $M\times M$ 
is a permissible filtered module with respect to the filtration 
$(F^*(M\times M))$.  Using the diagonal embedding $m\to (m,m)$ to identify $M$ with the diagonal of $M\times M$, we see that $(F^*(M))$ is the induced filtration of $(F^*(M\times M))$ on the diagonal, and therefore $(F^*(M))$ is a permissible filtration on $M$. 

Now, using the exact sequence \eqref{eq:sesexample}, we have
\[\ell(M/F^i(M))\leq \ell\big((M\times M)/F^i(M\times M)\big)=\ell\big(M/G^i(M)\big)+\ell\big(M/H^i(M)\big).\]
By considering  associated Hilbert--Samuel polynomials, we deduce that 
\begin{equation}\label{eq:kdim1}
\dim\big(F^*(M)\big)\leq \max\big\{\dim\big(G^*(M)\big),\, \dim\big(H^*(M)\big)\,\}.\end{equation}

For inequalities in the other direction, note that  the inclusions $F^i(M)\subseteq G^i(M)$ and $F^i(M)\subseteq H^i(M)$ imply that
\[ \ell (M/G^i(M))\leq \ell(M/F^i(M))\quad \text{and}\quad \ell (M/G^i(M))\leq \ell(M/F^i(M)).\]
Hence, from the associated Hilbert--Samuel polynomials, we see that
\begin{equation}\label{eq:kdim2}
\dim (G^*(M))\leq \dim (F^*(M))\quad \text{and}\quad \dim (H^*(M))\leq \dim (F^*(M)).\end{equation}
Combining the displayed inequalities \eqref{eq:kdim1} and \eqref{eq:kdim2} above, we obtain
\[\dim (F^*(M))=\dim (G^*(M))=\dim (H^*(M)),\]
and this completes the proof of Proposition \ref{P:kdim}.\end{proof}

The following corollary is now immediate from Note \ref{N:induced} and the result that the dimension is independent of  choice of permissible filtration.
\begin{corollary}\label{C:dimses} Let $R$ be a complete local-filtered ring,  let $M$ be a permissible $R$-space, and let  $L\subseteq M$ be an $R$-subspace of $M$.  Then $L$ is a permissible space and $\dim(L)\leq \dim(M)$.  Furthermore, if $L$ is (topologically) closed, then $M/L$ is a permissible space and $\dim(M)=\max\{\,\dim(L),\dim(M/L)\,\}$.
\end{corollary}

\subsection{Size of permissible spaces}

We will now specialise to the case of finite residue field,  and   prove the following result on the asymptotic size of spaces.

\begin{theorem}\label{T:asymptotic} Let $(R,\m)$ be a complete local-filtered ring with finite residue field,  let $M$ be a permissible $R$-space, and set 
\[
q:=|\, R/\m\, |,\quad \delta:=\dim(M).\]
Also, choose a permissible filtration  $(F^*(M))$  on $M$, and let $\alpha\in \Q$ be the leading coefficient of the Hilbert--Samuel polynomial of $(F^*(M))$. Then 
\begin{equation}\label{eq:limit}
\lim_{n\to\infty}\frac{\log_{q}\left\vert M/\m^nM\right\vert }{n^\delta}=\alpha.\end{equation} Consequently, the Hilbert--Samuel polynomial of any permissible filtration on $M$ is of the form $\alpha x^\delta +\text{lower degree terms}$.
\end{theorem} 
\begin{proof}
Using Theorem \ref{T:AR3} and arguing as in the displayed inclusions \eqref{eq:pertop}, we can find a non-negative integer $d$ such that
$F^{n+d}(M)\subseteq \m^nM\subseteq F^n(M) $ for all $n\in \N_0$. The cardinalities of the quotients then satisfy
\begin{equation} \label{eq:cards}\log_q\left\vert M/F^n(M)\right\vert \leq \log_{q}\left\vert M/\m^nM\right\vert \leq \log_q\left\vert M/F^{n+d}(M)\right\vert .\end{equation}
Let $\alpha x^\delta +Q(x)$ be the Hilbert--Samuel polynomial associated to the filtration $(F^*(M))$. Thus $Q(x)$ has degree at most $\delta-1$, and $\alpha$ is a positive rational number. Since 
$\log_q\left\vert M/F^i(M)\right\vert$ is the length of the $R$-space $M/F^i(M)$, the displayed inequalities \eqref{eq:cards} imply that
\[
\alpha n^\delta + Q(n)\leq \log_q\left\vert M/\m^nM\right\vert \leq \alpha(n+d)^\delta +Q(n+d) \] for sufficiently large $n$. The theorem is now immediate.\end{proof}

\subsection{Torsion elements in permissible spaces} \label{ss:torsion}
The preceding discussion can be applied to identify classes of torsion elements in permissible spaces over a complete local-filtered ring with finite residue field. The class of elements which we consider  is covered by the following definition.
 
 \begin{definition}\label{D:special elements}
Let $(R,\m)$ be a complete local-filtered ring, and let $M$ be an $R$-space. 
 An element $x\in M$ is said to be \emph{ $R$-distinguished}, or simply \emph{distinguished},  if $\m^i(Rx)=\m^ix$ for all integers $i\in \N_0$. If the stronger condition $\m^i(\m^jx)=\m^{i+j}x$ holds for all $i,j\in \N_0$, then we say that $x$ is \emph{$\m$-adically distinguished}.
\end{definition}

\begin{example}\label{E:nucleus} Let $(R,\m)$ be a complete local-filtered ring and let $M$ be an $R$-space. 
Then any element of $M$ on which $R$ operates associatively is $R$-distinguished. More precisely, if  $m\in M$ and  $r(sm)=(rs)m$ for all $r,s\in R$, then $m$ is $\m$-adically distinguished. In particular, if $R$ is associative and $M$ is an $R$-module, then every element of $M$ is $\m$-adically distinguished.
\end{example}

\begin{note}\label{N:DE} Let $(R,\m)$ be a complete local-filtered ring and let $M$ be an $R$-space. Then the following holds: \emph{An element $x\in M$ is $\m$-adically distinguished if and only if $Rx$ has a permissible filtration with cyclic associated gradation. }

The forward implication holds by definition. Conversely, if $Rx$ has a permissible filtration with cyclic associated gradation, then we must have $\gr(Rx)=\gr(R)\sigma(x)$ 
where $\sigma(x)$ is the principal part of $x$. Hence    $x$ is $\m$-adically distinguished (by  Example \ref{E:CS} \ref{CS2}).
\end{note}

\begin{theorem}\label{T:torsion}
Let $(R,\m)$ be a complete local-filtered ring with finite residue field and let $M$ be a permissible $R$-space. Assume that $\gr(R)$ is an integral domain, and that $M$ is  spanned by $\m$-adically distinguished elements with non-zero annihilators.  Then the following statements are equivalent.
\begin{enumerate}[label=(S\arabic*),   ]
\item \label{torsion1} $\dim(M)\leq \dim(R)-1$.
\item \label{torsion2} Every distinguished element of $M$ has a non-zero annihilator.
\item \label{torsion3} $M$ is  spanned by $\m$-adically distinguished elements with non-zero annihilators.
\end{enumerate}
\end{theorem}

\begin{remark}  Theorem \ref{T:torsion}   implies the following in an associative setting.  
\begin{enumerate}[wide=0pt]
\item[]\emph{Suppose $(R,\m)$ is an associative complete local Noetherian  ring with finite residue field and associated graded ring a commutative domain, and let $M$ be finitely generated  $R$-module. We assume that $M=Rx_1+\dots +Rx_n$ with each $x_i$ having non-zero annihilator. Then every element of $M$ has a non-zero annihilator}. \end{enumerate}
The conclusion is not obvious: there is no clear way of producing an annihilator of $x_1+x_2 $ out of relations $r_1x_1=r_2x_2=0$. \end{remark}

Returning to the proof of Theorem \ref{T:torsion}, we recall the following result from commutative algebra which plays  a significant role in the ensuing discussion. 
\begin{theorem}[height--dimension formula]\label{T:htdim} If $A$ is an affine domain and $I$ is an ideal of $A$, then   $\h(I)+\dim(A/I)=\dim(A)$.\end{theorem}

\begin{proof} See \cite[Ch. 5, 14.H]{matsCA}. The formula stated there (\emph{loc. cit.}) covers prime ideals. The result for  a general ideal $I$ follows on choosing a prime ideal $\p\supseteq I$ with $\h(\p)=\h(I)$, and then applying the cited result together with the inequalities $\dim(A/\p)\leq \dim(A/I)$ and $\h(I)+\dim(A/I)\leq \dim(A)$.\end{proof}

\begin{proof}[Proof of Theorem \ref{T:torsion}]The implication $\text{\ref{torsion2}}\implies \text{\ref{torsion3}}$ is immediate from the hypothesis that $M$ is  spanned by $\m$-adically distinguished elements with non-zero annihilators. The implications $\text{\ref{torsion1}}\implies \text{\ref{torsion2}}$ and $\text{\ref{torsion3}}\implies \text{\ref{torsion1}}$ are covered by Lemma \ref{L:torsion} and Lemma \ref{L:torsion converse} below, respectively. \end{proof}

\begin{lemma}\label{L:torsion} Let $(R,\m)$ be a complete local-filtered ring with finite residue field and let $M$ be a permissible $R$-space.
 If $\dim(M)\leq \dim(R)-1$, then every $R$-distinguished element in $M$ has a non-zero annihilator.
\end{lemma}

\begin{proof} Let $x\in M$ be $R$-distinguished. Then $Rx$ is an $R$-subspace of $M$. By Corollary \ref{C:dimses}, the $R$-space $Rx$ is permissible and $\dim(Rx)\leq \dim(M)$. In particular, $\dim(Rx)\neq \dim(R)$. Suppose now, for a contradiction, that $x$ has no non-zero annihilator. Then the map $R\to Rx$ defined by $r\to rx$ is an ismorphism of abelian groups. Moreover, from the defining property of being $R$-distinguished, we have $R/m^n\cong Rx/\m^n(Rx)$ for all integers $n\in \N_0$. By Theorem \ref{T:asymptotic},  we must have $\dim(R)=\dim(Rx)$, giving a contradiction.
\end{proof}

\begin{lemma}\label{L:torsion converse}
Let $(R,\m)$ be a complete local-filtered ring with finite residue field and let $M$ be a permissible $R$-space. Assume that $\gr(R)$ is an integral domain and that $M$ is  spanned by $\m$-adically distinguished elements with non-zero annihilators.  Then $\dim(M)\leq \dim(R)-1$. \end{lemma}

\begin{proof} Suppose $x\in M$ is $\m$-adically distinguished and killed by a non-zero element $r$ in $R$.  Since  $\gr(Rx)=\gr(R)\sigma(x)$ (from Note \ref{N:DE}),  we see that  the principal part $\sigma(r)$ is a non-trivial annihilator of the $\gr(R)$-module $\gr(Rx)$, and hence $\text{ht}(\ann\, \gr(Rx))\geq 1$. Hence we obtain $\dim(Rx)\geq \dim(R)-1$ from the height--dimension formula. 

Using the strong Noetherian property if necessary, we can find finitely many $\m$-adically distinguished elements $x_1,\dots ,x_n$ with non-zero annihilators such that $M=Rx_1+\dots +R x_n$. It then suffices to show that  \[\dim(M)=\max\{\dim(Rx_1),\dots ,\dim(Rx_n)\}.\] 
The desired equality  follows from Corollary \ref{C:dimses} and the following observations: 
\begin{enumerate}
\item[{$\vcenter{\hbox{\tiny$\bullet$}}$}] $Rx_1,\dots ,Rx_n$ are $R$-subspace of $M$;
\item[{$\vcenter{\hbox{\tiny$\bullet$}}$}] $\dim(Rx_1\oplus \dots \oplus Rx_n)=\max\{\dim(Rx_1),\dots ,\dim(Rx_n)\}$;
\item[{$\vcenter{\hbox{\tiny$\bullet$}}$}]  the map $Rx_1\oplus \dots \oplus Rx_n\to M$ given by $(r_1x_1,\dots ,r_nx_n)\to r_1x_1+\dots +r_nx_n$ is a surjective morphism of $R$-spaces. \qedhere \end{enumerate}
\end{proof}

\section{Central extensions and torsion elements in permissible spaces}\label{S:TPS}

Motivated by the theory of pseudo-null modules over Iwasawa algebras, we shall now consider implications of the preceding discussion   for  central extensions of complete local-filtered rings.  The setting  is follows. Given a central extension $R\subseteq S$ of complete local-filtered rings (that is,  $S$ is topologically generated over $R$ by central elements of $S$) and a space $M$ over $S$, determine elements of $M$ which are killed by $R$ (under some assumptions on the $S$-space $M$, of course).

 It will be convenient at this point to record an  assumption which occurs as part of the hypotheses in a number of propositions that follow. But first, a definition.  We say that an element $a\in R$ \emph{operates centrally} on the $R$-space $M$ if 
 \[ (ar)m=a(rm)=r(am)=(ra)m\] for all $r\in R$, $m\in M$.

\begin{assumption}\label{A:CE} $(R,\m)$ is a complete local-filtered ring with finite residue field;  $M$ is an $R[[T]]$-space on which  $T$ operates centrally, and $M$ is permissible as a space over $R$. \end{assumption}

The central extensions that we consider will be spaces over formal power series rings, and they arise as follows. Let $(R,\m)$ be a complete local-filtered ring with finite residue field, and let $M$ be a permissible space over $R$.  Suppose we are given a morphism $\phi : M\to M$ of $R$-spaces such that the induced linear map $M/\m M\to M/\m M$ is nilpotent.   Now $R$ is compact (as it is profinite), and consequently $M$ is also compact. We can therefore extend the $R$-space structure on $M$ to an $R[[T]]$-space structure by making $T$ operate centrally on $M$ via $\phi$.  In subsequent subsections, we will show that the construction produces a permissible $R[[T]]$-space (Proposition \ref{P:CEP}) and that its dimension is independent of the operator ring (Theorem \ref{T:CED}), and finally obtain an identification of classes of $R$-torsion elements (Theorem  \ref{T:CET}).

\subsection{Central extensions and permissibility}\label{ss:CEP} 
In this subsection, we show that   Assumption \ref{A:CE} already implies  permissibility  over the power series ring. 
\begin{proposition}\label{P:CEP} Let $(R,\m)$ and  $M$ be as in Assumption \ref{A:CE}.  Then 
$M$ is a permissible $R[[T]]$-space.
\end{proposition}
 
The proof of Proposition \ref{P:CEP} requires the following two lemmas.

\begin{lemma}\label{L:central1} Let $M$ be an $R$-space. Suppose $a\in R$ is invertible, central, and operates centrally on $M$. Then the inverse of $a$ also operates centrally on $M$.
\end{lemma}

\begin{proof} Let $b\in R$ be the inverse of $a$: that is, we have $ab=ba=1$. Then $b$ is also in the centre of $R$. It is easy to check that if $r\in R$ and $m\in M$, then we have equalities
\[ a((br)m)=a(b(rm))=a(r(bm))=a((rb)m)=rm.\]
The lemma then follows from the observation that  map $M\to M$ given by $m\to am$ is injective.
\end{proof}

\begin{lemma} \label{L:central2}  Let $(R,\m)$ and  $M$ be as in Assumption \ref{A:CE}.  Given a positive  integer  $j$, we can find a positive integer $k$ such that $T^kM\subseteq\mathfrak{m}^jM$.
\end{lemma}

\begin{proof}  Since $T(\m^jM)=\m^j(TM)\subseteq\m^jM$ and  $M/\m^j M$ is finite (by Proposition \ref{P:permissibletop}), it is enough to show the following: if $v\in M$, then $T^kv\in \m^jM$ for some positive integer $k$. 

So let us consider an element $v\in M$. Then we must have $T^{n+k}v\equiv T^k v\pmod{\m^j M}$ for some positive integers $k$ and $n$. Thus $(1-T^n)(T^k v)\in \m^j M$. Since $1-T^n$ is central, invertible and operates centrally, Lemma \ref{L:central1} implies that $T^kv\in \m^j M$.  \end{proof}

\begin{proof}[Proof of Proposition \ref{P:CEP}] Let $X$ be a central indeterminate, and consider the formal power series  objects $R[[X]]$ and $M[[X]]$ with their canonical filtrations (see Example \ref{E:filpoly}). It is then easy to see that $R[[X]]$ is complete  local-filtered, and that $M[[X]]$ is a permissible $R[[X]]$-space. 

We view $M$ as an $R[[X]]$-space by making $X$ operate as $T$ and consider the map $\theta: M[[X]]\to M$ given by 
\[\theta(m_0+m_1X+\dots)=m_0+Tm_1+\dots .\] 
Since the topology on $M$ is induced by the filtration $(\m^*M)$, Lemma \ref{L:central2} implies that the map $\theta$ is a well defined  homomorphism of $R[[X]]$-spaces. Also given a positive integer $k$, we can find---by Theorem \ref{T:AR3} and Lemma \ref{L:central2}---an integer $N$ such that 
$F^n(M)\subseteq \m^kM$ and $T^nM\subseteq \m^kM$ for all $n\geq N$. Thus if $i,j\in\N_0$ and $i+j\geq 2N$, then  
\[\theta(F^i(M)X^j)=T^jF^i(M)\subseteq \m^kM.\]
Now each subgroup $\m^nM$ is topologically closed (because $M/\m^nM$ is finite), and    the filtration on $M[[X]]$ is given by  $F^{n}(M[[X]])=\sum F^i(M)X^j$ where the summation runs over non-negative integers $i,j$ with $i+j\geq n$ (see Example \ref{E:filpoly}). It follows that
\[\theta(F^{2N}(M[[X]]))=\sum_{i+j\geq 2N}T^jF^i(M)\subseteq \m^kM, \]
and therefore the map  $\theta:M[[X]]\to M$ is continuous. Consequently, the  $R[[T]]$-space $M$ is permissible and  $(\,\theta(F^i(M[[X]]))\mid  i=0,1,\dots)$ is a permissible filtration on $M$ (see Note \ref{N:induced}). \end{proof}

\subsection{Central extensions and invariance of dimension}

Proposition \ref{P:CEP} raises the possibility of two different ways of calculating the dimension of a permissible space, either over the base ring or the central extension. The following result asserts that both calculations will give the same dimension.

\begin{theorem}\label{T:CED}  Let $(R,\m)$ and  $M$ be as in Assumption \ref{A:CE}.  Then the dimension of $M$ as an $R$-space is equal to the dimension of the $R[[T]]$-space $M$.\end{theorem}

We will first proof Theorem \ref{T:CED} under assumptions on the  $\m$-adic filtration, and then derive the general case using the Artin--Rees lemma.
\begin{proposition}\label{P:CED}
Let $(R,\m)$ and  $M$ be as in Assumption \ref{A:CE}. We also assume that  $F^i(M)=\m^iM$ is a permissible filtration, and that $\m^i(\m^jM)=\m^{i+j}M$ for all $i,j\in \N_0$.
Then the dimension of $M$ as an $R$-space is equal to the dimension of the $R[[T]]$-space $M$.
\end{proposition}

\begin{proof} Let $\n$ be the maximal ideal of $R[[T]]$. We write $d_{\m}$  and $d_{\n}$ for the dimension of $M$ as an $R$-space and $R[[T]]$-space, respectively. By Theorem \ref{T:asymptotic}, we can find constants $a,b\in \Q$  such that
\begin{equation}\label{eq:kdimpnull1}
\log_q\vert M/\m^kM\vert\sim ak^{d_\m}\quad\text{and}\quad  \log_q\vert M/\n^{k}M\vert \sim bk^{d_\n}.\end{equation}
Since $\m^k\subseteq \n^k$,  we have $\vert M/\n^kM\vert \leq \vert M/\m^k M\vert$ and so $d_\n\leq d_\m$.

For inequality in the other direction, we first fix a positive integer $N$ such that $T^NM\subseteq \m M$. This is possible by Lemma \ref{L:central2}. A straightforward  inductive argument using Theorem \ref{T:CLF} then shows that 
\begin{equation}\label{eq:TkN} T^{kN}M\subseteq \m^kM\end{equation} for all positive integers $k$. 

Now let $k$ be a positive integer, and let $n$ be a non-negative integer such that $n<kN$. We can then write $n=iN+j$ where the integers $i,j$ satisfy
$0\leq i<k$ and $0\leq j<N$. Using the inclusion \eqref{eq:TkN} and Theorem \ref{T:CLF}, we then have 
\[ (T^n\m^{kN-n})M=\m^{kN-n}(T^nM)\subseteq \m^{kN-n}(T^{iN}M)\subseteq \m^{kN-n}(\m^iM)=\m^{kN-n+i}M.\]
A stimple calculation shows that  $kN-n+i\geq  k$, and so  $(T^n\m^{kN-n})M\subseteq \m^kM$ for $0\leq n<kN$.  Since  
\[\n^{kN}=\m^{kN}+\m^{kN-1}T+\dots +\m T^{kN-1}+T^{kN}R[[T]],\] we obtain $\n^{kN}M\subset \m^k M$. So $\log_q\vert M/\m^kM\vert \leq \log_q\vert M/\n^{kN}M\vert$ for all positive integers $k$. The displayed estimates \eqref{eq:kdimpnull1} then imply that 
$ak^{d_\m}\leq b(Nk)^{d_\n}$ for $k\gg 0$. Hence $d_\m\leq d_\n$.
\end{proof}

\begin{proof}[Proof of Theorem \ref{T:CED}] By Theorem \ref{T:AR3}, there is a positive integer $D$ such that $F^{n+d}(M)=F^n(R) F^d(M)$ for all integers $n\in \N_0$ and $d\geq D$. Using Theorem \ref{T:CLF}, we see that $(\m^iF^D(M)\mid i=0,1,\dots)$ is a permissible filtration on $F^D(M)$ and that
\[\m^i(\m^j F^D(M))=\m^{i+j}F^D(M)\]
for all $i,j\in \N_0$.

Since $M$ is a strongly Noetherian $R$-space, the ascending chain of $R$-subspaces
\[ F^D(M)\subseteq F^D(M)+ TF^D(M)\subseteq F^D(M)+TF^D(M)+T^2F^D(M)\subseteq \dots\]
stabilises, say from $N:= F^D(M)+\dots +T^kF^D(M)$ onwards. Then, since $N$ is a closed $R$-subspace of $M$   and $TN\subseteq N$, Lemma \ref{L:central2} implies that $N$ is an $R[[T]]$-space. Furthermore, as $T$ operates centrally, the filtration 
$(\m^iN\mid i=0,1,\dots)$ is a permissible filtration on the $R$-space $N$ and 
$\m^i(\m^jN)=\m^{i+j}N$ for all $i,j\in \N_0$. Hence, by Proposition \ref{P:CED}, the dimension of the $R$-space $N$ is equal to the dimension of the $R[[T]]$-space $N$.

Finally, since $M/N$ has finite cardinality, the dimensions of $M$ and $N$ are equal as $R$-spaces, or as $R[[T]]$-spaces. Hence the dimension of $M$ as an $R$-space is the same as the dimension of the $R[[T]]$-permissible space $M$. 
\end{proof}

\subsection{Central extensions and torsion}   

Our aim is  to relate  presence of torsion in permissible spaces admitting a central extension to a property of associated gradations. 
To this end,  let $R$ be a complete local-filtered ring, and define a \emph{pseudo-null filtration} on an $R$-space $M$ to be a permissible filtration on $M$ such that the associated gradation is a pseudo-null $\gr(R)$-module. We recall that for a commutative ring~$A$, a
finitely generated~$A$-module~$N$ is said to be \emph{pseudo-null} if it satisfies one of the following equivalent conditions.
\begin{enumerate}[label=(PN\arabic*),   ]
\item \label{PN1} If $\p\in\spec(A)$ has height at most $1$, then the localisation ~$N_\mathfrak{p}=(0)$.
\item \label{PN2} If $\p\in\spec(A)$ contains the annihilator of $N$, then  $\text{ht}(\p)\geq 2$.
\end{enumerate}
See, for instance, \cite[Definition 5.1.4 ]{NSW}.  If $A$ is an affine domain, then  the pseudo-nullity of $N$ is equivalent to the following condition:
\begin{enumerate}[label=(PN\arabic*),   ]\setcounter{enumi}{2}
\item \label{PN3} $\dim(A)-\dim(N)\geq 2$. 
 \end{enumerate} The equivalence follows from the height--dimension formula (Theroem \ref{T:htdim}) and condition \ref{PN2}.

\begin{theorem}\label{T:CET} Let $(R,\m)$ and  $M$ be as in Assumption \ref{A:CE}. Assume that $M$ is  spanned by $\m$-adically distinguished elements as an  $R$-space,  and that $\gr(R)$ is an integral domain. Then the following statements are equivalent.
\begin{enumerate}[label=(T\arabic*),   ]
\item \label{CET2} Every $R$-distinguished element of $M$ has a non-zero annihilator in $R$.
\item \label{CET3} The $R[[T]]$-space $M$ has a pseudo-null filtration.
\item \label{CET4} All permissible filtrations on the $R[[T]]$-space $M$ are pseudo-null filtrations.
\end{enumerate}
\end{theorem}

\begin{proof} Note that  $M$ is a permissible $R[[T]]$-space by  Proposition \ref{P:CEP}.   Since
the dimension of a permissible space is independent of the choice of permissible filtration (Proposition \ref{P:kdim}), condition  \ref{PN3} implies the equivalences
\[\text{\ref{CET3}}\iff \text{\ref{CET4}} \iff \dim(R[[T]]-\dim(M)\geq 2.\]
Now, the dimension of $M$ is independent of the operator ring (Theorem \ref{T:CED}). Also, since $\gr(R[[T]])\cong \gr(R)[T]$ (see Example \ref{E:filpoly}, equation \eqref{eq:monfilpoly}), we have $\dim(R[[T]])=\dim(R)+1$. Thus \[\dim(R[[T]]-\dim(M)\geq 2\iff
\dim(R)-\dim(M)\geq 1.\]  The result now follows from Theorem \ref{T:torsion}.
\end{proof}

\begin{remark} In the classical setting of modules over commutative rings, Theorem \ref{T:CET} revisits the equivalence between pseudo-null $R[[T]]$-modules and torsion $R$-modules.  The theory of pseudo-null modules in the general associative setting is due to Venjakob (\cite{venjakob}), and it can be shown that the existence of a pseudo-null filtration with certain regularity conditions imply  pseudo-nullity of the module. It is then tempting to speculate on a structure theory for permissible spaces over complete local-filtered rings; that story though is outside the scope of this paper.
\end{remark}

\end{document}